\documentclass[reqno]{amsart}
\usepackage{hyperref}
\begin{document}
\title[elastodynamics model]{Initial boundary value problem for a nonconservative
system in elastodynamics}
\author[Kayyunnapara  Divya Joseph and  P. A  Dinesh]{Kayyunnapara  Divya Joseph and  P. A  Dinesh}
\address{
Kayyunnapara  Divya Joseph and P.A. Dinesh \newline
Department of Mathematics\\
MSRIT P. O\\
Bangalore 560054, India}
\email{divyakj@msrit.edu, dineshdpa@msrit.edu}
\thanks{ submitted }
\subjclass[2000]{35A20, 35L50,35R05}
\keywords{elastodynamics, viscous shocks}

\begin{abstract}
This paper is concerned with the initial boundary value
problem for a nonconservative system of hyperbolic
equation appearing in elastodynamics in the space time domain $x>0,t>0$.
The number of boundary conditions to be prescribed at the boundary $x=0$,
depend on the number of characteristics entering the domain. Since our system is nonlinear
the characteristic speeds depends on the unknown and the
direction of the characteristics curves are known apriori. As it is well known, 
the boundary condition has to be understood in a generalised
way. One of the  standard way is using vanishing viscosity method.
We use this method to construct solution for a particular class of initial and boundary data,
namely the initial and boundary datas that lie on the level sets of one of the Riemann invariants.
\end{abstract}

\maketitle
\numberwithin{equation}{section}
\numberwithin{equation}{section}
\newtheorem{theorem}{Theorem}[section]
\newtheorem{remark}[theorem]{Remark}

\section{Introduction} 
In this article we study a  system of two equations, appearing in elastodynamics,  namely,
\begin{equation}
\begin{gathered}
\frac{\partial u}{\partial t }+ u \frac{\partial u}{\partial x} - \frac{\partial{\sigma}}{\partial x} =
0,\\
\frac{\partial{\sigma}}{\partial t} + u \frac{\partial{\sigma}}{\partial x} - k^2
\frac{\partial{u}}{\partial x} = 0.
\end{gathered}
\label{e1.1}
\end{equation}
Here $u$ is the velocity, $\sigma$ is the stress and $k>0$
is the speed of propagation of the elastic waves and the system corresponds to the coupling of a dynamical law and a Hooke's law in a homogeneous medium in which density varies slightly in the neighborhood of a constant value.
Infact, this system was derived in \cite{c1} as a simplification of 
a system of four equations  of elasticity describing mass conservation, momentum conservation, energy conservation and Hooke's law.

Initial value problem for  the system \eqref{e1.1} 
is well studied by many authors \cite{c1,c2,j2,j3}, however the boundary value problem is not understood well. The aim of this paper is to study initial boundary value problem for \eqref{e1.1}, in $x>0,t>0$ with initial 
conditions
\begin{equation}
( u(x,0),\sigma (x, 0) ) = (u_0(x) , \sigma _0(x) ),\,\,\,
x > 0
\label{e1.2}
\end{equation}
and a weak form of the  boundary conditions,
\begin{equation}
( u(0,t),\sigma (0,t) ) = (u_b(t) , \sigma _b(t)
),\,\,\,t>0.
\label{e1.3}
\end{equation}
We write \eqref{e1.1} as the system
\begin{equation*}
\begin{bmatrix}
u\\  
\sigma
\end{bmatrix} 
_t +
\begin{bmatrix}
u & -1\\
-k^2 & u
\end{bmatrix}
\begin{bmatrix}
u \\
\sigma
\end{bmatrix} 
_x
=0
\end{equation*}
Let 
\begin{equation*}
A(u,\sigma) =
\begin{bmatrix}
u & -1\\
-k^2 & u
\end{bmatrix}
\end{equation*}
The system \eqref{e1.1} is  non-conservative because
it cannot be written in the conservative form, namely
\begin{equation*}
\frac {\partial u}{\partial t} + \frac {\partial f_1(u, \sigma)}{\partial x} = 0, \,\,\,
\frac {\partial \sigma}{\partial t} + \frac {\partial f_2(u, \sigma)}{\partial x} = 0.
\end{equation*}
In other words there does not exists $f: R^2 \rightarrow
R^2$ such that
\begin{equation*}
A(u,\sigma) =\frac{\partial( f_1,f_2)}{\partial
(u,\sigma)}
\end{equation*}
where $f(u,\sigma)=(f_1(u,\sigma),f_2(u,\sigma))$.
The eigenvalues of $A(u,\sigma)$ are called the
characteristic speeds of system \eqref{e1.1}. A simple
computation
shows that, the equation for the eigenvalues are given
by\begin{equation*}
\det (\lambda I -
\begin{bmatrix}
u & -1\\
-k^2 & u
\end{bmatrix} )
= 0
\end{equation*}
This gives
\begin{equation*}
(\lambda - u )^2  - k^2 =0
\end{equation*}
which has two real distinct roots and so the system is strictly hyperbolic. We call $\lambda_1$ the
smaller one and $\lambda_2$ larger one, then we have
\begin{equation}
\lambda_1 = u-k ,\,\,\,\lambda_2 = u + k.
\label{e1.4}
\end{equation} 
The curves defined by $\frac{dx}{dt}= \lambda_1$ and  $\frac{dx}{dt}= \lambda_2$ are called  the $1-$characteristic curve and $2-$characteristic curve respectively. 
Call $\ r_1 $ and $\ r_2 $ right eigenvectors
corresponding to $ \lambda_1 $ and $ \lambda_2 $
respectively. An
 easy computation shows
 \begin{equation*}
r_1 = r_1(u,\sigma)=(1,k),\,\,\,r_2= r_2 (u,
\sigma)=(1,-k)
\end{equation*}
We say $w_i :R^2 \rightarrow R$ is a  $i$-Riemann
invariant if 
\begin{equation*}
Dw_{i}(u,v) . r_i(u,\sigma)=0, 
\end{equation*}
where $ D w_i(u,v)=(w_{iu},w_{i \sigma})$, is the  the gradient of $w_i$ w.r.t. $(u,\sigma)$ and $r_i(u,\sigma)$   is a right eigen vector corresponding to the $i$-th characteristic speed 
$\lambda_i(u,\sigma)$.
First we prove few basic properties of Riemann invariants. For the special system that we are studying , these two families of Riemann invariants  can be constructed 
globally and they play an important role in our analysis. \\

{\bf Lemma: } The functions $w_1$ and $w_2$ given by 
\begin{equation}
w_1(u,\sigma) =\sigma - k u,\,\,\,
w_2(u,\sigma) =\sigma + k u 
\label{e1.5}
\end{equation}
are $1-$ Riemann invariant and $2-$Riemann invariant
respectively. Further $w_1$ is constant along $2-$ characteristic curve
and $w_2$ is constant along $1-$ characteristic curve.\\
\begin{proof}:
By the definition , $w_1(u,\sigma)$ is a $1-$ Riemann invariant
if
$w_{1 u} + k w_{1 \sigma} = 0$ and $w_2(u,\sigma)$ $2-$ Riemann invariant if
$w_{2 u} -k w_{2 \sigma} = 0$.
Now  $w_{1  u} + k w_{1 \sigma} = 0$ gives
$w_1(u,\sigma)=\phi_1(\sigma - k u) $, where $\phi_1 :R^1
\rightarrow R^1$ is arbitrary smooth function.
whereas $w_{2 u} -k w_{2 \sigma} = 0$. gives
$w_2(u,\sigma)=\phi_2(\sigma + k u)$, $\phi_2:R^1\rightarrow
R^1$. In particular taking $\phi(w)=w$, we get the pair of Riemann invariants as given in Lemma.
 An easy computation shows that
\begin{equation*}
{w_1}_t= \sigma_t - k u_t ,\,\,\,
{w_1}_x= \sigma_x - k u_x ,\,\,\,
{w_2}_t= \sigma_t + k u_t ,\,\,\,
{w_2}_x= \sigma_x + k u_x 
\end{equation*}
This together with \eqref{e1.1}, after rearranging rearranging the terms we get ,
\begin{equation*}
\begin{gathered}
\frac{\partial w_1}{ \partial t} + \lambda_2 \frac{
\partial w_1}{ \partial x} =
\sigma_t - k u_t + (u+ k)( \sigma_x - k u_x)\\ 
=( \sigma_t + u \sigma_x - k^2 u_x)- k(u_t +u u_x - \sigma_x )= 0,\\
\frac{\partial w_2 }{\partial t} + \lambda_1 \frac{
\partial w_2 }{ \partial x} = 
\sigma_t + k u_t + (u- k)( \sigma_x + k u_x) \\
=( \sigma_t + u \sigma_x - k^2 u_x)+ k(u_t +u u_x - \sigma_x )= 0.
\end{gathered}
\end{equation*}
In short,
\begin{equation*}
\frac{\partial w_1}{ \partial t} + \lambda_2 \frac{
\partial w_1}{ \partial x} = 0, \qquad
\frac{\partial w_2 }{\partial t} + \lambda_1 \frac{
\partial w_2 }{ \partial x} = 0.
\end{equation*}

Let $x_i=\beta_i(t)$ be an $i$-characteristic curve . We have
\begin{equation*}
\begin{gathered}
\frac{d w_1}{d t}(\beta_2(t),t)= \frac{\partial w_1}{\partial t} + \frac{\partial \beta_2}{\partial t} \frac{\partial w_1}{\partial x} = \frac{\partial w_1}{\partial t} + \frac{d x_2}{dt} \frac{\partial w_1}{\partial x} =  \frac{\partial w_1}{\partial t} + \lambda_2  \frac{\partial w_1}{\partial x} = 0 , \\
\frac{d w_2}{ d t}(\beta_1(t),t)= \frac{\partial w_2}{\partial t} + \frac{\partial \beta_1}{\partial t} \frac{\partial w_2}{\partial x} = \frac{\partial w_2}{\partial t} + \frac{d x_1}{dt} \frac{\partial w_2}{\partial x} =  \frac{\partial w_2}{\partial t} + \lambda_1 \frac{\partial w_2}{\partial x} = 0 . \\
\end{gathered}
\end{equation*}
which proves the lemma.
\end{proof}

\section{Initial boundary value problem for the linearized system}

In this section we study initial boundary value problem for the linearized system of \eqref{e1.1} around a constant state $(\bar{u},\bar{\sigma})$. Let
\[
u=\bar{u} +\epsilon u(x,t) +...,\,\,\,\sigma(x,t)=\bar{\sigma} +\epsilon \sigma(x,t) +...
\]
where $(\bar{u},\bar{\sigma})$ is constant. Substituting this in\eqref{e1.1} and equating  powers of $\epsilon$
we get
waves, namely,

\begin{equation}
\begin{gathered}
\frac{\partial u}{\partial t }+ \bar{u} \frac{\partial u}{\partial x} - \frac{\partial{\sigma}}{\partial x} =
0,\\
\frac{\partial{\sigma}}{\partial t} + \bar{u} \frac{\partial{\sigma}}{\partial x} - k^2
\frac{\partial{u}}{\partial x} = 0,
\end{gathered}
\label{e2.1}
\end{equation}

Writing this in Riemann invariants we get
\begin{equation}
\frac{\partial w_1}{ \partial t} + (\bar{u}+k) \frac{
\partial w_1}{ \partial x} = 0 \qquad
\frac{\partial w_2 }{\partial t} + (\bar{u}-k) \frac{
\partial w_2 }{ \partial x} = 0
\label{e2.2}
\end{equation}
Whose general solution is
\[
w_1(x,t)=\phi_1(x-(\bar{u}+k)t),\,\,\,w_2(x,t)=\phi_2(x-(\bar{u}-k)t)
\]
for any functions $\phi_i: R \rightarrow R^2, i = 1, 2$.

Now if we consider the inirial boundary value problem in $x>0,t>0$,
with initial condition
\[
u(x,0)=u_0(x),\,\,\,\sigma(x,0)=\sigma_0(x)
\]
and boundary condition
\[
u(0,t)=u_b(t),\,\,\,\sigma(0,t)=\sigma_b(t).
\]
In  Riemann invariant they take the form 
\begin{equation*}
w_1(x,0)=w_{10}(x)=\sigma_0(x) - ku_0(x) ,\qquad
w_2(x,0)=w_{20}(x)=\sigma_0(x) + ku_0(x)
\end{equation*}
and
\begin{equation*}
w_1(0,t)=w_{1b}(t)=\sigma_b(t)-ku_b(t) ,\qquad
w_2(0,t)=w_{2b}(t)=\sigma_b(t)+ku_b(t)
\end{equation*}
respectively.
The boundary conditions at $x=0$, can be prescribed only if  both  $\bar{u}+k$ and $\bar{u}-k$ are
positive. 

To understand the situation, consider the equation
\[
 a _t + \lambda a_x = 0
 \]
where $ \lambda $ is constant. The general solution is to
this PDE is $ a = \phi (x - \lambda t) $ which is the
unique solution of this PDE with
initial condition $a(x,0)=\phi(x)$.
Now  if we look at the
initial boundary value problem in the quarter plane
$x>0,t>0$, initial data at $t=0$ uniquely
determine the solution in $x>0,t>0$ if $\lambda<0$ because
any characteristic curve through $(x,t), x>0,t>0$ can be
drawn back to the initial line i.e the $x$-axis
at a point $y=x-at$ and $u(x,t)=u(x-at,0)$. So no boundary
conditions are required at $x=0$.  On the other hand,
if $\lambda > 0$ , then all the characteristic lines at $x
= 0$ are pointing inwards the domain $x>0,t>0$, and then the initial data , $a(x,0)=a_0(x), x>0$ determine the solution 
in the region $x>\lambda t$.  The boundary condition $a(0,t)=a_b(t)$, is required to determine the solution
in $0<x<\lambda t$. So $a(0, t)$ and $a(x, 0)$ should be given to determine the solution in $x>0,t>0$.
In fact we can write down the solution explicitly.

{\bf Case 1 $\lambda <  0$ :} Here no boundary condition is required.
The characteristic curve is  $x = \lambda t +y$ ,
$a(x,t)= a_0(x-\lambda t)$  \\

{\bf Case 2: $\lambda > 0$ : } 
Boundary condition $a(0,t)=a_b(t)$  is required.The characteristic curve is  $x = \lambda t +y$ .

If $x\geq\lambda t$ , we have $a(x,t)=a_0(x-\lambda t)$.
If $x<\lambda t$ , $a(x,t)=a(0,\tau)$ where $(0,\tau)$ is the point of intersection of the characteristic passing through $(x,t)$ with slope $\lambda$ and the $t$-axis.
So
$a(x,t)=a_b(\tau)=a_b(t-(x/\lambda))$
\begin{equation}
a(x,t)=\begin{cases}
a_0(x-\lambda t) &  \text{if  }  x\geq \lambda t, \\
a_b(t-(x/\lambda)) &  \text{if  }   x< \lambda t
\end{cases}
\label{e2.3}
\end{equation}

Now getting back to the system \eqref{e2.1}, or equivalently \eqref{e2.2}, the following cases occur.\\
{\bf Case 1:  $\bar{u}+k <0$ :} . In this case the initial data  $w_1(x,0)=w_{10}(x)$,  $w_2(x,0)=w_{20}(x)$
determine the solution and no conditions are required on $x = 0$ and explicit solution in $x>0,t>0$ is given by
\[
w_1(x,t)= w_{10}(x-(\bar{u}+k)t ) \qquad
w_2(x,t)= w_{20}(x-(\bar{u} -k)t )
\]

{\bf Case 2:  $\bar{ u}- k < 0$ and $\bar{u} + k >0$ :} Here, apart from the initial conditions as above,  one condition is required. This is of the form 
\begin{equation*}
\alpha u(0,t) +\beta \sigma(0,t) =\gamma(t), 
\end{equation*}
where $\gamma(t)$ is a given function of $t$ and $\alpha , \beta$ are constants with $(k\beta - \alpha) \ne 0$.
In Riemann invariants, this is equivalent to 
\begin{equation*}
  \alpha \frac{(w_2-w_1)(0,t)}{2k} +\beta \frac{(w_1+w_2)(0,t)}{2} = \gamma(t). 
\end{equation*}

So, we have
\[
(k \beta + \alpha )w_2(0,t)+  (k \beta - \alpha )w_1(0,t)=2k \gamma(t).
\] 
Since $\bar{ u}- k < 0$ as before,
\begin{equation*}
w_2(x,t)= w_{20}(x-(\bar{u} -k)t ) , x>0, t>0.
\end{equation*}
and from \eqref{e2.3} 
\begin{equation*}
w_1(x,t)=\begin{cases}
 w_{10}(x-(\bar{u}+k)t ),&\text{if } x\geq (\bar{u}+ k)t \\
w_1(0,t-(\frac{x}{\bar{u}+k})),&\text{if } x< (\bar{u}+ k)t \\
\end{cases}
\end{equation*}
which imples
\begin{equation*}
w_1(0,t -( \frac{x}{ \bar{u} +k }))= 
 \frac{2 k }{(k\beta - \alpha)}2 k \gamma (t - ( \frac{x}{ \bar{u} +k })) - \frac{ (k \beta + \alpha) }{(k\beta - \alpha)}w_{20}( ( k - \bar{u})(t - ( \frac{x}{ \bar{u} +k }))
\end{equation*}
Rewriting we get
\begin{equation*}
w_1(x,t)= 
\begin{cases}
 w_{10}(x-(\bar{u}+k)t ) & \text{if } x\geq(\bar{u}+ k)t \\
 \frac{2 k }{(k\beta - \alpha)}2 k \gamma (t - ( \frac{x}{ \bar{u} +k })) - \frac{ (k \beta + \alpha) }{(k\beta - \alpha)}w_{20}( ( k - \bar{u})(t - ( \frac{x}{ \bar{u} +k })) & \text{if } x<(\bar{u}+ k)t , \\
\end{cases}
\end{equation*}
\begin{equation*}
w_2(x,t)= w_{20}(x-(\bar{u} -k)t ) .
\end{equation*}

{\bf Case 3: 
$\bar{u} - k > 0.$ } This implies $\bar{u} + k >0.$
So, the conditions $\sigma(0,t) = \sigma_b(t)$, $u(0,t)=u_b(t)$ should be prescribed.
More generally one can prescribe,
\begin{equation}
\alpha_{11} u(0,t) +\beta_{11} \sigma(0,t) =\gamma_1(t) ,\,\,\,
\alpha_{22} u(0,t) +\beta_{22} \sigma(0,t) =\gamma_2(t)
\label{e2.4} 
\end{equation}
with $\det 
\begin{bmatrix}
\alpha_{11} &  \beta_{11} \\
\alpha_{22} &   \beta_{22}
\end{bmatrix}
 \neq 0$. Now,
\eqref{e2.4}
in Riemann invariants is equivalent to
\begin{equation}
(\alpha_{11}+k\beta_{11} )w_2(0,t) +(k\beta_{11}-\alpha_{11}) w_1(0,t) =2k \gamma_1(t)
\label{e2.5} 
\end{equation}
and
\begin{equation}
(\alpha_{22}+k\beta_{22} )w_2(0,t) +(k\beta_{22}-\alpha_{22}) w_1(0,t) =2k \gamma_2(t).
\label{e2.6}
\end{equation}

From \eqref{e2.3}, we have
\begin{equation*}
w_1(x,t)=\begin{cases}
 w_{10}(x-(\bar{u}+k)t ),&\text{if } x\geq (\bar{u}+ k)t \\
w_1(0,t-(\frac{x}{\bar{u}+k}),&\text{if } x< (\bar{u}+ k)t \\
\end{cases}
\end{equation*}
and
\begin{equation*}
w_2(x,t)=\begin{cases}
 w_{20}(x-(\bar{u}-k)t ),&\text{if } x\geq (\bar{u}- k)t \\
w_2(0,t-(\frac{x}{\bar{u}-k}),&\text{if } x< (\bar{u}- k)t. \\
\end{cases}
\end{equation*}
Note that \eqref{e2.5} and \eqref{e2.6} can be written as,
\begin{equation*}
\begin{bmatrix}
\alpha_{11}+k\beta{11} &  k\beta_{11}-\alpha{11} \\
\alpha_{22}+k\beta{22} &   k\beta_{22}-\alpha{22}
\end{bmatrix}
\begin{bmatrix}
w_2(0,t) \\
w_1(0,t)
\end{bmatrix}
=
\begin{bmatrix}
2k \gamma_1(t) \\
2k \gamma_2(t),
\end{bmatrix}
\end{equation*}
which can be solved if
\begin{equation*}
det 
\begin{bmatrix}
\alpha_{11}+k\beta_{11} &  k\beta_{11}-\alpha_{11} \\
\alpha_{22}+k\beta_{22} &   k\beta_{22}-\alpha_{22}
\end{bmatrix}
\neq 0.
\end{equation*}
An easy computation shows that this is equivalent to,
\begin{equation*}
(\alpha_{11}+k\beta_{11})  ( k\beta_{22}-\alpha_{22}) - (\alpha_{22}+k\beta_{22}) ( k\beta_{11}-\alpha_{11}) 
\neq 0
\end{equation*}
which is equivalent to,
\begin{equation*}
2k (\alpha_{11} \beta_{22} -  \beta_{11}\alpha_{22}) \neq 0,
\end{equation*}
i.e
\begin{equation*}
\alpha_{11} \beta_{22} -  \beta_{11}\alpha_{22} \neq 0,
\end{equation*}
which we have assumed earlier.
A special Case is
$\alpha_{11}=1, \beta_{11}=0$
$\alpha_{22}=0, \beta_{11}=1$
which is the Dirichlet's Boundary condition
 $u(0,t)=\gamma_1(t),
\sigma(0,t)=\gamma_2(t).$

\section{Initial boundary value problem when the data is
on the level set of Riemann invariants}

 The system \eqref{e1.1} is nonconservative and strictly
hyperbolic with characteristic speeds
$\lambda_1(u,\sigma) = u - k, \lambda_2(u,\sigma) = u + k$
and it is well known that smooth global in time
solutions does not exist even if the initial data
is smooth. Since the system is nonconservative any
discussion of well-posedness of solution should be based
on a given nonconservative product in addition to the
admissibility criterion for shock discontinuities.
Indeed the system is an approximation and is obtained when
one ignores higher order terms, which give smoothing
effects with
small parameters as coefficients . So the physical
solution is constructed as the limit of a given
regularization as these parameters
goes to zero. Different regularizations correspond to
different nonconservative product and admissibility
condition. In this paper we take a parabolic
regularization, 
\begin{equation}
\begin{gathered}
u_t + u u_x - \sigma_x = \epsilon u_{xx},\\
\sigma_t + u \sigma_x - k^2  u_x = \epsilon \sigma_{xx}.
\end{gathered}
\label{e3.1}
\end{equation}
This approximation is particularly useful in the study of the initial boundary value problem for \eqref{e1.1} in $x>0,t>0,$ as we cannot prescribe  initial and boundary conditions,
\begin{equation}
u(0,t)=u_b(t), \sigma(0,t) = \sigma_b (t),u(x,0)= u_0(x), 
\sigma(x,0) = \sigma_0(x),
\label{e3.2}
\end{equation}
for \eqref{e1.1} in the strong sense.

In this section we consider the initial boundary value
problem for \eqref{e1.1} with initial condition \eqref{e1.2} and a weak 
form of the boundary condition \eqref{e1.3} when
initial and boundary data lie on the level set of one of
the Riemann invariants. We take the data
to be on the level set of the $j-$  Riemann invariant and therefore
\begin{equation} 
\sigma _0 (x)+(-1)^{j} k u_0(x) = c , \,\,\,\sigma _b(t) +(-1)^{j} k
u_b(t) = c
\label{e3.3}
 \end{equation} 
Throughout this section we assume that $u_0 , \sigma_0$ are bounded measurable and $u_b , \sigma_b$ are continuous.
We look for a solution of the form $ \sigma =-(-1)^{j} k u + c$
substituting this in \eqref{e1.1} we get 
\[
 u_t + u u_x -(-1)^{j+1} k u_x = 0,\,\,,\,
k( -(-1)^{j}) k u _t + u ( k(-(-1)^{j}) k u _x ) - k^2  u_x = 0 .
\]
This gives 
\[
 u_t + (u -(-1)^{j+1} k) u_x = 0. 
 \]
Making the substitution $v= u -(-1)^{j+1} k$ the problem
\eqref{e1.1} reduces to
\begin{equation}
v_t +v v_x = 0, 
 v( x, 0) = u_0 -(-1)^{j+1} k , 
 v(0,t) = u_b(t) -(-1)^{j+1} k .
\label{e3.4}
\end{equation}
The boundary condition $v( 0, t) = u_b(t) - k$
cannot be satisfied in general, in the strong sense as the
characteristics depends on the solution and at the boundary point the
characteristics may point out of the boundary when drawn
in increasing time direction. We
construct the solution $u(x,t)$ as the limit of
$u_\epsilon (x,t)$ as $\epsilon \rightarrow 0$ for \eqref{e3.1}
with initial and boundary conditions
\eqref{e3.2} 

To describe the vanishing viscosity limit, we need to introduce some notations.
First we introduce a functional
defined on a
certain  class of paths, that is described in the following way.
For each fixed $(x,y,t), x \geq 0, y \geq 0, t>0$,
$C(x,y,t)$ denotes the following class of paths $\beta$ inthe
quarter plane
$D=\{ (z,s) : z\geq 0, s \geq 0\}$. Each path is connected
from the point
$(y,0)$ to $(x,t)$ and is of the form $z=\beta(s)$, where
$\beta$ is a
piecewise linear function of maximum three lines. On
$C(x,y,t)$, we
define a functional
\begin{equation}
J(\beta) = -\frac{1}{2}\int_{\{s:\beta(s)=0\}}
((u_b(s)+(-1)^{j}k)^{+})^2 ds
+ \frac{1}{2} \int_{\{s:\beta(s) \neq 0\}} 
\left(\frac{d\beta(s)}{ds}\right)^2 ds.
\label{e3.5}
\end{equation}

We call $\beta_0$ is straight line path connecting $(y,0)$
and $(x,t)$
which does not touch the boundary $x=0$,
namely
$\{(0,t), t>0\}$, then let
\begin{equation}
 A(x,y,t)= J(\beta_0) = \frac{(x-y)^2}{2t}.
\label{e3.6}
\end{equation} 
Any $\beta \in C^{*}(x,y,t) = C(x,y,t)-\{\beta_0\}$
is made up of three pieces of lines
connecting $(y,0)$ to $(0,\tau_1)$ in the interior and
$(0,\tau_1)$ to
$(0,\tau_2)$ on the boundary and $(0,\tau_2)$ to $(x,t)$ in the
interior, or two pieces or one line.

For such curves, it can be easily seen that
\begin{equation}
J(\beta) = J(x,y,t,\tau_1,\tau_2) = 
-\int_{\tau_1}^{\tau_2}\frac{((u_b(s)+(-1)^{j}k)^{+})^2}{2}ds + 
\frac{y^2}{2 \tau_1} + \frac{x^2}{2(t-\tau_2)}.
\label{e3.7}
\end{equation}
For the case where $\tau_1 = 0,y=0,$
\begin{equation*}
J(\beta) = J(x,0,t,0,\tau_2) = 
-\int_{\tau_1}^{\tau_2}\frac{((u_b(s)+(-1)^{j}k)^{+})^2}{2}ds +  \frac{x^2}{2(t-\tau_2)}.
\end{equation*}

It was proved in \cite{j4} that there exists $\beta^{*}\in C^{*}(x,y,t)$
and corresponding $\tau_1(x,y,t),\tau_2(x,y,t)$ so that
\begin{equation}
\begin{aligned}
B(x,y,t)&= \min \{J(\beta) :\beta \in C^{*}(x,y,t)\}\\
&= \min \{J(x,y,t,\tau_1,\tau_2): \,\, 0\leq \tau_1 < \tau_2 < t\}\\
&=J(x,y,t,\tau_1(x,y,t),\tau_2(x,y,t)).
\end{aligned}
\label{e3.8}
\end{equation} 
is Lipshitz continuous. Further, it was also proved that,
\begin{equation}
\begin{aligned}
Q(x,y,t)&= \min\{J(\beta) : \beta \in C(x,y,t)\}\\
        & = \min \{A(x,y,t),B(x,y,t)\}
\end{aligned}
\label{e3.9}
\end{equation}
and
\begin{equation}
U(x,t)= \min \{Q(x,y,t) + U_0(y), \,\, 0\leq y< \infty\}
\label{e3.10}
\end{equation}
are Lipshitz continuous function in their variables,
where
\[
U_0(y)=\int_0^y u_0(z) dz +(-1)^{j} k y .
\]
  Further minimum in
\eqref{e3.10}
is attained at some value of $y\geq0$, which depends on
$(x,t)$, we
call it $y(x,t)$. If $A(x,y(x,t),t)\leq B(x,y(x,t),t)$
\begin{equation}
\begin{aligned}
U(x,t)= \frac{(x-y(x,t))^2}{2t} +U_0(y(x,t))
\end{aligned}
\label{e3.11}
\end{equation}
and if If $A(x,y(x,t),t)> B(x,y(x,t),t)$
\begin{equation}
U(x,t)=J(x,y(x,t),t,\tau_1(x,y(x,t),t),\tau_2(x,y(x,t),t))+U_0(y(x,t))
\label{e3.12}
\end{equation}

Here and hence forth $y(x,t)$ is a minimizer in
\eqref{e3.10} and in
the case of \eqref{e3.12}, $\tau_2(x,t)=\tau_2(x,y(x,t),t)$ and
$\tau_1(x,t)=\tau_1(x,y(x,t),t)$.
With these notations, we have an explicit formula for the
solutions of
\eqref{e1.1} and \eqref{e1.2} with a weak form of the boundary condition
\eqref{e1.3} ,as the vanishing viscocity limit of \eqref{e3.3}- \eqref{e3.4}
 given by the following result.

\begin{theorem} For every $(x,t),x>0,t>0$ minimum in
\eqref{e3.10}
is achieved by some $y(x,t)$ (may not be unique) and
$U(x,t)$ is a
Lipschitz continuous. Further the limit $\lim_{\epsilon}(u^\epsilon, \sigma^\epsilon)=(u(x,t),\sigma(x,t))$ exists and has the following representation. For almost every $(x,t),$ there exists unique
minimum $y(x,t)$
and either 
$A(x,y(x,t),t)< B(x,y(x,t),t)$ or
$B(x,y(x,t),t)<A(x,y(x,t),t)$. Define
\begin{equation}
p(x,t) = \begin{cases}
 \frac{x-y(x,t)}{t},&\text{if } A(x,y(x,t),t)< B(x,y(x,t),t),\\
 \frac{x}{t-t_1(x,t)},&\text{if } A(x,y(x,t),t)> B(x,y(x,t),t).
\end{cases}
\label{e3.13}
\end{equation}
Then 
\begin{equation}
u(x,t)=p(x,t)+(-1)^{j+1}k,\,\,\,
\sigma(x,t) = (-1)^{j+1} k(p(x,t)+(-1)^{j+1}k)+c 
\label{e3.14}
\end{equation}
Further $(u(x,t),\sigma(x,t))$ is a 
weak solution of 
of \eqref{e1.1}, with initial conditions 
\eqref{e1.2}.
\end{theorem}
\begin{proof} We try to find a solution which
lies on the level set of the Riemann invariant satisfying the initial and boundary conditions.
When $\sigma=(-1)^{j+1} k u +c$, the system \eqref{e3.3}
become a single Burgers
equation
\[
 u_t +u u_x - (-1)^{j+1}ku_x =\epsilon u_{xx},x>0,t>0
\]
with initial  conditions
\begin{equation}
u(x,0)=u_0(x),\,\,\,x>0 
\label{e3.15}
\end{equation}

and boundary conditions
\begin{equation}
 u(0,t)=u_b(t), t>0.
\label{e3.16}
\end{equation}
A more general type of initial boundary value problem of
this type was studied by Bardos, LeRoux and Nedelec \cite{b1} and they
proved the convergence of the limit as $\epsilon$ goes to $0$ . They further showed that the limit satisfies the equation 
\[
 u_t +u u_x - (-1)^{j+1}ku_x =0 ,x>0,t>0 ,
\] 
the initial condition \eqref{e3.15} and a weak form of the boundary condition \eqref{e3.16} , namely 
\begin{equation}
u(0+, t) \in 
\mathcal{E}(u_b(t))=\begin{cases}
{u_b} \bigcup (- \infty , - u_b(t)) & \text{if } u_b(t) > 0 ,\\
(- \infty , 0 ] & \text{if } u_b(t) \leq  0.
\end{cases}
\end{equation}

 Here our aim is to
get the formula for the limit.  So
taking $v=u-(-1)^{j+1}k $, the equation can be written as
\[
 v_t +v v_x =\epsilon v_{xx}, x>0, t>0
\]
\[
v(x,0)=u_0(x)-(-1)^{j+1}k, x>0,\,\,\, v(0,t)=u_b(t)-(-1)^{j+1}k,t>0.
\]
Applying Hopf-Cole transformation  \cite{h1} 
\begin{equation*}
 v=-2\epsilon (\log w)_x,
\end{equation*}
the problem is reduced to
\begin{equation*}
 w_t=\epsilon w_{xx}
\end{equation*}
with initial conditions
\begin{equation*}
w(x,0)=e^{\frac{-1}{2\epsilon}(\int_0^x u_0(z)dz
-(-1)^{j+1}kx)}
\end{equation*}
and boundary conditions
\[
2\epsilon w_x(0,t)+(u_b(t)-(-1)^{j+1}k)w(0,t)=0.
\]
Existence of solution $w(x,t)$ for this problem is well known and from the 
transformations we get,
\begin{equation}
\begin{aligned}
&u^\epsilon(x,t)=v^\epsilon(x,t)+(-1)^{j+1}k =- 2 \epsilon \frac{w^\epsilon_x(x,t)}{w^\epsilon(x,t)}+(-1)^{j+1}k,\\
&\sigma^\epsilon(x,t)=(-1)^{j+1} k u^\epsilon(x,t) +c =(-1)^{j+1} k( - 2 \epsilon \frac{w^\epsilon_x(x,t)}{w^\epsilon(x,t)}+(-1)^{j+1}k ) +c.
\label{e3.18}
\end{aligned}
\end{equation}
The limit of $ p(x,t) =\lim_{\epsilon \rightarrow 0} \{- 2 \epsilon \frac{w^\epsilon_x(x,t)}{w^\epsilon(x,t)}\}$ exists as shown in \cite{b1}. Further it was shown in
\cite{j1,j4} that, with $A$ and $B$ as in the theorem
\begin{equation}
p(x,t) = \begin{cases}
 \frac{x-y(x,t)}{t},&\text{if } A(x,y(x,t),t)<B(x,y(x,t),t),\\
 \frac{x}{t-\tau_1(x,t)},&\text{if }  A(x,y(x,t),t)> B(x,y(x,t),t) ,
\end{cases}
\label{e3.19}
\end{equation}
The formula then follows from \eqref{e3.18} and \eqref{e3.19}.

For the completeness of the arguement,  we show that the limit satisfies the equation\eqref{e1.1}.
For this first we note that the system is equivalent to
\[
 u_t +u u_x - (-1)^{j+1}ku_x =0, x>0, t>0
\]
in weak sense, when $(u,\sigma)$ lies on one of the Riemann invariant. This follows easily because, if
$\sigma(x,t)=(-1)^{j+1}k u(x,t) +c$,
\[
\begin{aligned}
&u_t+uu_x-\sigma_x = u_t +u u_x - (-1)^{j+1}ku_x\\
&\sigma_t+u\sigma_x-k^2 u_x = (-1)^{j+1}\{u_t +u u_x - (-1)^{j+1}ku_x\}.
\end{aligned}
\]
Now since $u^\epsilon$ is smooth it satisfies, 
\[
\begin{aligned}
 -\int_0^\infty \int_0^\infty (u^\epsilon (x,t) \phi_t(x,t)  &+ (\frac{u^\epsilon(x,t)^2}{2} - (-1)^{j+1}k u^\epsilon(x,t))\phi_x(x,t))dx dt \\ = \epsilon \int_0^\infty \int_0^\infty u^\epsilon \phi_{xx} dx dt
\end{aligned}
\]
for any smooth function with compact support in $x>0,t>0$.

Since $u_0$ is bounded measurable and $u_b$ is continuous, by the maximum principle we have that,  $u^\epsilon$ is bounded in $[0,\infty) \times [0,T]$ and pointwise convergent almost everywhere. So by passing to the limit as $\epsilon$ tends to zero and
using  the Dominated convergence theorem , we get
\[
-\int_0^\infty \int_0^\infty (u (x,t) \phi_t(x,t)  +  (\frac{u(x,t)^2}{2} - (-1)^{j+1}k u(x,t))\phi_x(x,t))dx dt = 0.
\]
This completes the proof.
\end{proof}

{\bf REMARK:}  The limit in the theorem satisfies the initial conditions 
\begin{equation*}
\begin{aligned}
u(x,0)=u_0(x),\,\,\,\sigma(x,0)=\sigma_0(x),\,\,\, x>0
\end{aligned}
\end{equation*} 
but  the boundary conditions at $x=0$,
\begin{equation*}
\begin{aligned}
u(0,t)=u_b(t),\,\,\,\sigma(0,t)=\rho_b(t).
\end{aligned}
\end{equation*}
is not satified in the strong sense. Indeed
with strong form of Dirichlet boundary conditions, there is
neither existence nor uniqueness as the speed of
propagation
$\lambda_j(u,\sigma) =u+(-1)^jk$ does not have
a definite sign at the boundary $x=0$. We note that the
speed is
completely determined by the first equation. However the boundary conditions
are satisfied in the sense of
Bardos Leroux and Nedelec \cite{b1} 
which for our case is
equivalent to the following condition, see LeFloch
\cite{le1}
\begin{equation*}
\begin{gathered}
\text{either } u(0+,t) -(-1)^{j+1}k= (u_{b}(t)-(-1)^{j+1}k )^{+}\\
\text{or } u(0+,t) -(-1)^{j+1}k\leq 0 \text{ and }
(u(0+,t)-(-1)^{j+1}k)^2\leq [(u_{b}(t) -(-1)^{j+1}k)^{+}]^2 
\end{gathered}
\end{equation*}
where
$u_{b}^{+}(t)= \max\{u_b(t),0\}$.

Also the solution satisfies the entropy condition $u(x-,t)\geq u(x+,t)$ which follows from the formula derived in the theorem.

\section{Initial boundary value problem with Riemann type boundary conditions}

In this section, we find an explicit formula for global
solution of
\eqref{e1.1} when the initial data \eqref{e1.2} and the boundary
data \eqref{e1.3} are of Riemann type.

We have the following formula for the vanishing viscosity
limit.

\begin{theorem}
Let $u^\epsilon$ and $\sigma^\epsilon$ are given by 
\eqref{e3.18}, with $u_0,\sigma_0,u_b,\sigma_b$
all constants and lie on the level sets of the $j$
Riemann invariants then $u(x,t)= \lim_{\epsilon
\rightarrow 0} u^\epsilon(x,t)$
exists pointwise a.e. and $\sigma(x,t)= lim_{\epsilon
\rightarrow 0}
\sigma^\epsilon(x,t)$, in the sense of distributions and
$(u,\sigma)$
have the following form.

{Case 1: $u_0 -(-1)^{j+1}k =u_b -(-1)^{j+1}k >0$,} 
\[
(u(x,t),\sigma(x,t)) =
(u_0-(-1)^{j+1}k,k[-(-1)^{j+1}k+u_0]+c) ,\,\,\,if \,\,\,x
>0.
\]

{Case 2: $u_0 -(-1)^{j+1}k=u_b -(-1)^{j+1}k<0$,}
\[
(u(x,t),\sigma(x,t) =
(u_0-(-1)^{j+1}k,k[-(-1)^{j+1}k+u_0]+c).
\]

{Case 3: $0<u_b -(-1)^{j+1}k<u_0-(-1)^{j+1}k$,}
\[
(u(x,t),\sigma(x,t)) = \begin{cases} \displaystyle
{(u_b-(-1)^{j+1}k,k[-(-1)^{j+1}k+u_b]+c),\,\,\, if
\,\,\,x<u_b t}
\\\displaystyle
{ (x/t-(-1)^{j+1}k, k( x/t-(-1)^{j+1}k )+c    ),\,\,\, if \,\,\,u_b t< x<u_0
t}\\\displaystyle{(u_0- (-1)^{j+1}k,k[-(-1)^{j+1}k+u_0]+c),\,\,\, if \,\,\,
x>u_0 t}.
\end{cases}
\]
{Case 4:  $u_b -(-1)^{j+1}k<0<u_0-(-1)^{j+1}k$,}
\[
(u(x,t),\sigma(x,t)) = \begin{cases} \displaystyle
{(x/t-(-1)^{j+1}k, k[-(-1)^{j+1}k+ x/t]+c),\,\,\, if
\,\,\,0< x<u_0 t}
\\\displaystyle
{(u_0- (-1)^{j+1}k,k[u_0- (-1)^{j+1}k]+c),\,\,\, if \,\,\,x>u_0 t}.\end{cases}
\]
{Case 5: $u_b-(-1)^{j+1}k<0$ and $u_0-(-1)^{j+1}k\leq 0$,}
\[
(u(x,t),\sigma(x,t)) = (u_0 -(-1)^{j+1}k,k[-(-1)^{j+1}k+u_0]+c).
\]
{Case 6: $u_0-(-1)^{j+1}k<u_b-(-1)^{j+1}k$ and $(u_b-(-1)^{j+1}k)+(u_0 -(-1)^{j+1}k)>0$ }
\[
(u(x,t),\sigma(x,t)) = \begin{cases} \displaystyle
{(u_b- (-1)^{j+1}k,k[-(-1)^{j+1}k+u_b]+c),\,\,\, if
\,\,\,x<st}
\\\displaystyle
{ (u_b-(-1)^{j+1}k,k[-(-1)^{j+1}k+u_b]+c),\,\,\, if
\,\,\,x>st},
\end{cases}
\]
where $s=\frac{u_0 +u_b }{2}-(-1)^{j+1}k$. This limit
satisfies the equation
\eqref{e1.1} and the initial condition \eqref{e1.2}.
\end{theorem}

\begin{proof}: 
For simplicity , we call $u_b -(-1)^{j+1}k = \bar{u_b},$ $u_0 -(-1)^{j+1}k = \bar{u_0}.$
Fix $(x,y,t)$, $x\geq 0, y\geq 0, t\geq 0$ . For $0\leq \tau_2\leq \tau_1<t$, 
consider the function 
\[
f(\tau_1,\tau_2) =\frac{x^2}{2(t-\tau_1)} +\frac{y^2}{2 \tau_2}-\frac{1}{2}{\bar{u_b}}^2(\tau_1 -\tau_2).
\]
We find the minimum with respect to $\tau_1$ and $\tau_2$. Note that as $\tau_1$ goes to $t$ or $\tau_2$ goes to $0$, $f$
goes to infinity if $x>0$ and $y>0$ respectively. So for this case the minimum is inside $0 <\tau_2\leq \tau_1<t$. 
Taking first derivative condition for mimnimum, we get
\[
\frac{\partial f}{\partial {\tau_1}}=\frac{x^2}{2(t-\tau_1)^2}-\frac{1}{2}{\bar{u_b}}^{2}=0,\,\,\,\frac{\partial f}{\partial \tau_2 }=-\frac{y^2}{\tau_2}^2
+\frac{1}{2} {\bar{u_b}}^2=0
\]
Solving this we get $\tau_1 =t-\frac{x}{\bar{u_b}}$,$ \tau_2 = \frac{y}{\bar{u_b}}$. For this we require $\bar{u_b} >0$ and $x<\bar{u_b}  t$

Case $0<\bar{u_b}<\bar{u_0}$ :
We analysis different regions . In the region $ 0<x<\bar{u_b} t$ , we have minimum point is  at $(\tau_1,\tau_2) = (t-\frac{x}{\bar{u_b}},\frac{y}{\bar{u_b}} )$ and so
\[
\begin{aligned}
B(x,y,t)&=\frac{x^2}{2(t-\tau_1)} +\frac{y^2}{2 \tau_2}-\frac{1}{2}{\bar{u_b}}^2(\tau_1 -\tau_2)\\
             &=\frac{x}{2} \bar{u_b} +\frac{1}{2} y \bar{u_b} -\frac{1}{2}{u_b}^2(t-\frac{x}{\bar{u_b}}-\frac{y}{\bar{u_b}})\\
             &=\bar{u_b}(x+y) -\frac{1}{2}\bar{u_b}^2 t
\end{aligned}
\] 
Now consider 
\[
A_{min}(x,t)=\min_{y\geq 0}\{\bar{u_0} y + \frac{(x-y)^2}{2 t}\}.
\] 
Denote $g(y)$ the function inside the bracket , the $g'(y)=\bar{u_0} -\frac{x-y}{t} <0 $ in the region $x-\bar{u_b} t <0$
so 
\[
A_{min}(x,t) =\frac{x^2}{2t}.
\]. Let
\[
B_{min}(x,t)=\min_{y \geq 0}\{\bar{u_0} y +\bar{u_b}(x+y) -\frac{1}{2} \bar{u_b}^2 t\}
\]
Since the function inside the bracket is an increasing function of $y$, minimum is achieved for $y=0$, so that
\[
B_{min}(x,t)= \bar{u_b} x -\frac{1}{2} \bar{u_b}^2 t
\]
Now $A_{min}(x,t)-B_{min}(x,t)=\frac{(x-\bar{u_b} t)^2}{2t}>0$
so we have
\[
U(x,t)=\min\{(A_{min}(x,t),B_{min}(x,t)\}= \bar{u_b} x -\frac{1}{2} \bar{u_b}^2 t
\]
Now consider the region $\bar{u_b} t <x< \bar{u_0} t$. In this region
\[
\begin{aligned}
\frac{\partial f}{\partial \tau_1} &=\frac{x^2}{2(t-\tau_1)^2} -\frac{1}{2} \bar{u_b}^2\\
&=\frac{1}{2}(\frac{x}{t-\tau_1}+\bar{u_b})(\frac{x- \bar{u_b} (t-\tau_1)}{t-\tau_1})
\end{aligned}
\]
which is positive. So as a function of $\tau_1$ $f$ is an increasing function and so its minimum is achived at $\tau_1 =0$. This means that $B_{min}(x,t)$ is achived at
$\tau_1=\tau_2=y=0$, so that
\[
B_{min}(x,t) =\frac{x^2}{2t}.
\]
As in previous case
\[
A(x,t)=\frac{x^2}{2t}.
\]
so that
\[
U(x,t)=\frac{x^2}{2t}.
\]
In the region $x>u_0 t$, as before
\[
B_{min}(x,t) =\frac{x^2}{2t}.
\]
From the first derivative condition, it is easy to see
that minimum in $A_{min}(x,t)$ is achieved for $y=x - \bar{u_0} t$ and so
\[
A_{min}(x,t)=\bar{u_0} x -\frac{1}{2}\bar{u_0}^2 t
\]
But $A_{min}(x,t)-B_{min}(x,t)= - \frac{(x- \bar{u_0} t)^2}{2t}<0$ and we get
\[
U(x,t)=\min\{(A(x,t),B(x,t)\}= \bar{u_0} x -\frac{1}{2} {\bar{u_0}}^2 t
\]

Case : $0<\bar{u_0}<\bar{u_b}$
As in the previous analysis, it follows that
\[
A_{min}(x,t)=\begin{cases}\bar{u_0} x -\frac{\bar{u_0}^2 t}{2},\,\,\,x>\bar{u_0} t\\
             \frac{x^2}{2 t},\,\,\, x<\bar{u_0} t 
\end{cases}
\]
and
\[
B_{min}(x,t)=\begin{cases}\bar{u_b} x -\frac{\bar{u_b}^2 t}{2},\,\,\,x<\bar{u_b} t\\
   \frac{x^2}{2 t},\,\,\, x>\bar{u_b} t 
\end{cases}
\]
Same analysis as before gives
\[
U(x,t)=\begin{cases}\bar{u_b} x -\frac{\bar{u_b}^2 t}{2},\,\,\,x<\bar{u_b} t\\
   u_0 x -\frac{\bar{u_0}^2 t}{2},\,\,\, x>\bar{u_0} t 
\end{cases}
\]
Now if $\bar{u_0} t <x<\bar{u_b} t$,
\[
\begin{aligned}
A_{min}(x,t)-B_{min}(x,t)= &=\bar{u_0} x -\frac{\bar{u_0}^2 t}{2}
-(\bar{u_b} x -\frac{\bar{u_0}^2 t}{2})\\
= (\bar{u_0} -\bar{u_b} )(x-\frac{1}{2}(\bar{u_0} +\bar{u_b})t)
\end{aligned}
\]
The right hand side is positive if 
$x<\frac{1}{2}(\bar{u_0} +\bar{u_b})t$ and negative if
$x>\frac{1}{2}(\bar{u_0}+\bar{u_b})t$ and we get
\[
U(x,t)=\begin{cases}\bar{u_b} x -\frac{\bar{u_b}^2 t}{2},\,\,\,u_0 t<x<\frac{\bar{u_0}+\bar{u_b}}{2} t\\
             u_0 x -\frac{\bar{u_0}^2 t}{2},\,\,\,
             \frac{\bar{u_0}+\bar{u_b}}{2} t <x< \bar{u_b} t 
\end{cases}
\]
Case $\bar{u_b} \leq 0$ . In this case 
\[
B(x,y,t)=\min_{0<\tau_2<\tau_1<t} \frac{x^2}{2(t-\tau_1)}
+\frac{y^2}{2 \tau_2}
\]
Here $\frac{\partial f}{\partial \tau_1}=\frac{x^2}{2(t-\tau_1)}>0$. So minimum is achived in $B_{min}(x,t)$
at $\tau_1=\tau_2=y=0$ so that $B_{min}(x,t)=\frac{x^2}{2 t}$. Now if $\bar{u_0}>0$,
\[
A_{min}(x,t) =\begin{cases} \frac{x^2}{2t}, \,\,\, x<\bar{u_0} t \\
 \bar{u_0} x -\frac{\bar{u_0}^2}{2} t,\,\,\, x>\bar{u_0} t.
\end{cases}
\]
so that if $\bar{u_0}>0, \bar{u_b}\leq0$, we have
\[
U(x,t) =\begin{cases} \frac{x^2}{2t},\,\,\,x>\bar{u_0} t\\
\bar{u_0} x-\frac{\bar{u_0}^2}{2}t,\,\,\,x<\bar{u_0} t
\end{cases}
\]
If $\bar{u_0}<0, \bar{u_b}\leq 0$, the minmum in $A_{min}(x,t)$ is
achieved at $y=x- \bar{u_0} t.$
 $A_{min}(x,t)=\bar{u_0} x -\frac{\bar{u_0}^2}{2}t$ 
and as before $B_{min}(x,t)=\frac{x^2}{2t}$, a we get
\[
U(x,t)=\bar{u_0} x -\frac{ \bar{u_0}^2}{2}t
\]
\end{proof}
This completes the proof.

\end{document}